\documentclass[letterpaper, 10 pt, conference]{ieeeconf}

\IEEEoverridecommandlockouts
\usepackage{graphics}
\usepackage{epsfig}
\usepackage{amsmath}
\usepackage{amssymb}
\usepackage{bbm}
\usepackage{eufrak}
\usepackage{subfig}
\usepackage[noadjust]{cite}

\newtheorem{lemma}{Lemma}
\newtheorem{theorem}{Theorem}

\newtheorem{definition}{Definition}

\newtheorem{remark}{Remark}

\usepackage{comment}
\usepackage{mathtools}
\usepackage{soul}
\usepackage{tikz}
\usepackage{color}
\usepackage{xcolor}
\usepackage{tcolorbox}
\usepackage{float}
\usepackage{subfig}
\usepackage{algorithm}
\usepackage[noend]{algpseudocode}
\usepackage[noabbrev]{cleveref}

\makeatletter
\def\BState{\State\hskip-\ALG@thistlm}
\makeatother

\renewcommand{\vec}[1]{\mathbf{#1}}
\graphicspath{{images/}}
\title{\LARGE \bf Linear Model Regression on Time-series Data: \\Non-asymptotic Error Bounds and Applications}

\graphicspath{{images/}}

\author{Atiye Alaeddini$^{\dagger}$, Siavash Alemzadeh$^{\ddagger}$, Afshin Mesbahi$^{\ddagger}$, and Mehran Mesbahi$^{\ddagger}$
	\thanks{$^{\dagger}$ Institute for Disease Modeling, 3150 139th Ave SE, Bellevue, WA, 98005, {\tt\small  aalaeddini@idmod.org}.
	A. Alaeddini thanks Bill \& Melinda Gates for their support and sponsorship through Global Good Fund.}
	\thanks{$^{\ddagger}$ William E. Boeing Department of Aeronautics and Astronautics, University of Washington, Seattle, WA, 98195-2400, {\tt\small \{alems,amesbahi,mesbahi\}@uw.edu}.
	The work of the last three authors has been supported by NSF SES-1541025.}
}

\begin{document}
	
	\maketitle
	\thispagestyle{empty}
	\pagestyle{empty}

	%========================================================================
	
	\begin{abstract}
		
		Data-driven methods for modeling dynamic systems have recently received considerable attention as they provide a mechanism for control synthesis directly from the observed time-series data.
		In the absence of prior assumptions on how the time-series had been generated, regression on the system model has been particularly popular.  In the linear case, the resulting least squares setup for model regression, not only provides a computationally viable method to fit a model to the data, but also provides useful insights into the modal properties of the underlying dynamics.
		Although probabilistic estimates for this model regression have been reported, deterministic error bounds have not been examined  in the literature, particularly as they pertain to the properties of the underlying system. 
			In this paper, we provide deterministic non-asymptotic error bounds for fitting a linear model to
			observed time-series data, with a particular attention 
		%We provide this for both cases of absence/presence of control input in the system and 
		to the role of symmetry and eigenvalue multiplicity in the underlying system matrix.
		%Finally, we will provide an optimization method to find the optimal control based on the observed data up to some specific time-step.
		
		% Then we focus on some specific dynamics where this identification method could possibly fail. Finally, we provide some verification via examples on related dynamical systems.}
		% Dynamic mode decomposition (DMD) has recently established as a data-driven method for control. In spite of the preliminary results on incorporation of control in DMD, the notion of optimal control and error bounds on system parameters estimation are missing in the literature. In this paper, we find an error bound for a specific case of online single data stack collection. Then an optimization problem is introduced to find the optimal initial condition. The same error bounds are provided for the controlled phase and also the optimal control is derived based on the available data. Finally, we provide some examples showing how the bounds work for real-world data.\\
		
%		\vspace{2mm}
		
		
	\end{abstract}
	
	%========================================================================
	\section{INTRODUCTION}
	\label{sec:intro}

	Recent advances in measurement and sensing technologies have lead to the availability of an unprecedented amount of data generated by complex physical, social, and biological systems such as turbulent flow, opinion dynamics on social networks, transportation, financial trading, and drug discovery.
	This so-called big data revolution has resulted in the development of efficient computational tools that utilizes the data generated by a dynamic system to reason about reduce order representations of this data, subsequently utilized for classification or prediction on the underlying model.
	Such techniques have been particularly useful when the derivation of models from first principles is prohibitively complex or infeasible.
In the meantime, utilizing data generated by the system directly for the purpose of control or estimation, poses a number of challenges, most notably for model-based control design techniques such as ${\cal H}_{2}$, ${\cal H}_{\infty}$, and model predictive control (MPC).
As such, it has become imperative to examine fundamental limits on fitting models to the time-series data, that can subsequently be used for model-based control synthesis.

One caveat of such an approach for a wide range of complex systems is the absence of the ability to excite the system with desired (persistent) inputs for the purpose of system identification~\cite{ljung1998system}. 
%
%System identification has been the cornerstones of system theory, where control and estimation-based on the observed input and output data-can be realized.
%%investigated throughout the literature history, in many different communities and under various nomenclature.
%Some of the earliest works on system identification goes back to Zadeh \cite{zadeh:56,zadeh1962circuit}, who
%refers to identification as a principal component of system theory for determining relevant system properties
%using input and output data.\footnote{Zadeh is credited with coining the name ``system identification" as well.}
%%
%System identification has been extensively developed in subsequent decades for the purpose of parameter estimation and
%adaptive control~\cite{ljung1998system,aastrom1973self}.
%	
	%In these methods, however, for perfect identification the input signals need to cover the whole space which is not guaranteed for every dynamical system. This condition is referred to as \textit{persistence of excitation} and is central to system identification and adaptive control \cite{narendra2012stable}.
	More recently, data-driven identification has also been examined in the context of machine learning as an extension of classification or prediction, with less attention given to the ability to excite the system with persistent inputs.
	In this direction, non-asymptotic bounds for finite sample complexity were obtained in \cite{hardt2016gradient,campi2002finite,oymak2018non,dean2017sample,simchowitz2018learning} for the linear time-invariant systems.
	Maximum likelihood and subspace identification methods have been employed in \cite{boots2009learning} to learn linear systems with guaranteed stability.
	The problem was investigated in an online learning setup to find regret bounds on the average cost of linear quadratic (LQ) systems in~\cite{abbasi2011regret}.
	In the context of data-driven analysis of dynamical systems, Koopman analysis has also been used for operator-theoretic identification of nonlinear systems and their spectral properties \cite{mauroy2016linear, de2017controllability, brunton2016koopman}.
	One of the key elements used in the aforementioned identification methods is a linear regression step in order to fit a model
	to data; linear regression is in fact one of the backbone of what is known as supervised learning.\footnote{Where a linear model is trained for labeling future instances of incoming data.}
	In its most basic form, linear regression is used to find the system parameters by solving a least-squares minimization problem constructed on the observed  time-series. Examples of such an approach can be found in \cite{fiechter1997pac,dean2017sample,parsa2018hierarchical}.

	%For example, in \cite{fiechter1997pac} least-squares algorithm is leveraged with observation-output pairs of data and then the estimated model is passed through the Ricatti equation to find the optimal control. The same holds for \cite{dean2017sample} where a probabilistic bound is found on the error in random design least-squares with vector-valued observations. Linear regression is also used in other learning methods such as in \cite{bradtke1994incremental} where recursive least-squares is used to update the Q function in the Q-learning algorithm (\textcolor{red}{other LS examples can be added}).
	
	The error analysis for fitting a linear dynamic system to data presented in this work is closely related to  error estimates examined in~\cite{fiechter1997pac} and \cite{dean2017sample} for linear quadratic synthesis. In both works, control synthesis involves an intermediate step of parameterizing the underlying system using the collected data; subsequently, probabilistic guarantees on the error between the true and the estimated models are presented.
		In \cite{fiechter1997pac}, the underlying system is allowed to be excited by canonical inputs before time-series data is collected following each ``episode".
	The same approach has been adopted in \cite{dean2017sample}, where a Gaussian noise is used to excite the system.
	While meaningful upper bounds on the error of the estimate are examined in these works, the presented results are probabilistic in nature, with probabilistic bounds that  are directly related to the number of data points.
	In this work, we provide non-asymptotic error bounds for adopting a regression approach to fit a linear model to data generated by the system, evolving from initial conditions and without a control input.
 Furthermore, this error bound is analyzed in an online (non-asymptotic) manner as more data becomes available.
 It is shown that the error guarantees are closely related to the system parameters, the rank of the collected data, and not surprisingly, the initial conditions.
 We then focus on the case where the underlying linear model is symmetric and show that the modeling regression error depends on the spectral properties of the system.
	
	%In particular, we will show that for symmetric systems, the least error we can get depends on the value of the largest eigenvalue with multiplicity greater than one.
	
	\color{black}

	The organization of the paper is as follows:
	In \S\ref{sec:math-prel}, we provide the necessary mathematical background.
	The problem setup is outlined in \S\ref{sec:prob-setup}.
	\S\ref{sec:analysis} provides the error analysis and main results of the paper.
	We then examine the ramification of our results for fitting a linear model to network data in \S\ref{sec:sims}.
	The paper is concluded in \S\ref{sec:conclusion}.
	
	%===========================================================================

	\section{Notation and Preliminaries}
	\label{sec:math-prel}
%	the set of nonnegative real numbers by $\mathbb{R}_+$; 
	The {cardinality} of a set $S$ is denoted as $|S|$.
	A column vector with $n$ real entries is denoted as $\vec{v} \in\mathbb{R}^n$, where $\vec{v}_i$ represents its $i$th element.
	The matrix $M\in \mathbb{R}^{p\times q}$ contains $p$ rows and $q$ columns and $[M]_{ij}$ denotes its (real)
	entry at the $i$th row and $j$th column.
	%; its rank will be denoted by \textit{\bf rank} of $M$
%	The \textit{\bf rank} of $M$ is defined as the number of its independent rows or columns; the matrix has full-rank when
%	$\text{\bf rank}(M)=\min\{p,q\}$.
	The Moore-Penrose pseudo-inverse of a full-rank matrix  $M\in \mathbb{R}^{p\times q}$ is defined as $M^{\dagger}=(M^\top M)^{-1}M^\top$ if $p>q$ and $M^{\dagger}=M^\top(MM^\top)^{-1}$ otherwise;
	$M^\top$ denotes the transpose of the matrix.
	%
%	The matrix $N\in\mathbb{R}^{n\times n}$ is \emph{symmetric} when $N^\top=N$.
	%
	The \emph{range} and \emph{nullspace} of matrix $M$ are denoted by $\mathcal{R}(M)$ and $\mathcal{N}(M)$, respectively.
	The \emph{basis} of the vector space $\mathcal{V}$ is referred to as $\mathcal{B}(\mathcal{V})$; the span 
	of a set of vectors, denoted by ${\textit{\bf span}}$, is the set of all linear combinations of these vectors. 
%Thus, for example, $\textit{\bf span}\{\mathcal{B}(\mathcal{V})\}=\mathcal{V}$.
	%
%	For $\vec{w}\in\mathbb{R}^n$, $\mathrm{\bf Diag}(\vec{w})$ is an $n\times n$ matrix with $\vec{w}$ on its diagonal and zero elsewhere.
	The unit vector $\vec{e}^i$ is the column vector with all zero entries except $[\vec{e}^i]_i=1$.
	The column vector of all ones is denoted by $\textbf{1}$, with dimension implicit from the context.
	The $n \times n$ identity matrix is defined as $I_n=\text{\bf Diag}(\textbf{1})$ for $\vec{1}\in\mathbb{R}^n$.
%	the matrix of all ones is then the product $\textbf{11}^\top$.
	The {trace} of $M\in\mathbb{R}^{n\times n}$ is designated as $\textit{\bf Tr} \text (M)=\sum_{i=1}^n [M]_{ii}=\sum_{i=1}^n \lambda_i$, where $\lambda_i$ is the $i$th eigenvalue of $M$.
	We write $M\succ 0$ when $M$ is positive-definite (PD) and	$M\succeq 0$ if $M$ is positive-semidefinite (PSD).
	% A matrix is positive definite if all of its eigenvalues are positive.
	The \textit{spectrum} of $M$ (the set of its eigenvalues) is denoted by $\mathbf{\Lambda}(M)$.
	The \textit{algebraic multiplicity} of an eigenvalue $\lambda$ is denoted by $m(\lambda)$, defined as the multiplicity of $\lambda$ in $\mathbf{\Lambda}(M)$.
	An eigenvalue $\lambda$ is called \emph{simple} if $m(\lambda)=1$.
	The singular value decomposition (SVD) of a matrix $X \in \mathbb{R}^{n\times m}$ is the factorization $X= {U} \Sigma {V}^\top$, where the unitary matrices ${U}$ and ${V}$ consist of the left and right ``singular" vectors of $X$, and $\Sigma$ is the diagonal matrix of singular values.
	%; the columns of ${U}$ are the eigenvectors of $XX^\top$
	Economic SVD is the reduced order matrix obtained by truncating the factor matrices in the SVD to the first $r$ columns of $U$ and $V$, corresponding to the $r$ non-zero singular values in $\Sigma$, where $r=\text{\bf rank}(X)$.
	From the SVD of a given matrix $X$, one can find its pseudo-inverse as $X^{\dagger}={V}\Sigma^{-1}{U}^\top$, resulting in $XX^{\dagger}={U}{U}^\top$; when $\Sigma$ has zero diagonals, the aforementioned inverse keeps these zeros untouched.
	The \textit{Euclidean norm} of a vector $\vec{x} \in\mathbb{R}^n$ is defined as $\| \vec{x} \|_2=(\vec{x}^\top \vec{x})^{1/2}=(\sum_{i=1}^n x_i^2)^{1/2}$.
%; we will use $\|\cdot\|$ as a substitute for $\|\cdot\|_2$ for vector norms.
	The \textit{spectral norm} of matrix $X$ is defined as $\|X\|_2=\sup\{\| X \vec{u} \|_2: \|\vec{u} \|_2=1\}$. 
%For a symmetric matrix $M$, $\|M\|_2=\lambda_{\max}(M)$, where $\lambda_{\max}(N)$ is the largest eigenvalue of $M$.
	The Frobenius norm of a matrix $\| X \|_F$ is defined as $\| X \|_F = \sqrt{\textit{\bf Tr} (X^\top X)}$.
	%
%	For both matrix norms we have the inequality $\|AB\|\leq \|A\| \|B\|$;
	Spectral and Frobenius norms satisfy the inequality $\|X\|_2 \leq \|X\|_F \leq r\|X\|_2$, where $r=\text{\bf rank}(X)$.
	
	%=========================================================================
	\section{Problem Setup}
	\label{sec:prob-setup}
	
	For a wide range of real-world systems, the underlying complex dynamics makes deriving the corresponding 
	models from first principles difficult if not infeasible.
	This can be due to a range of factors from the unpredictable nature of the environment to perturbations and uncertainties in the complex system \cite{alemzadeh2018influence,dean2017sample}.
	However, with the availability of sensing technologies and high performance computing, time-series data can be collected from these systems.
	Hence, it becomes natural to consider to what extent the observed time-series
	can be used to reason about the underlying dynamic model.
 In the case when this data has been generated by simulations (a model, albeit complex, already exists), 
 one might still be interested to reason about the dynamics using ``simple" models.
The adoption of this approach involves using prior knowledge of the underlying dynamics to choose a particular basis or library, and then postulating that the dynamic system is some combination of these basis elements. This problem then reduces to a parameter optimization problem -with respect to these basis or library- for their combination that best fits the given data, with respect to a suitable norm or metric.
In the absence of any prior assumption on the dynamics, however, it is often desirable to explore simple models.
%	
%	For example, one motivation of this can be finding the required amount of data for an estimation with a pre-specified error.
%	However, based on the application on which we implement the method, the error might be larger than some threshold making the methodology ineffective.
	This paper examines non-asymptotic error bounds for doing such a linear fit, for the case when the data had been generated by a linear system; generically, it is the case that if the data is rich enough and the system does not have degeneracies, exact model is obtained after the number of data snapshots is the dimension of the system.\footnote{The model has $n^{2}$ unknown entries; as such, $n^{2}$ observations are generically needed for its exact recovery. One of course can get away with less data by invoking sparsity (say, using the $\ell_1$ regularization) or structure on the model, e.g, assuming an underlying pattern for the system matrix.}
	In fact, we show that even in this most streamlined case, and even in the absence of noise or uncertainty in the collected data, understanding non-asymptotic behavior of the error requires some non-trivial analysis. 
%	Such an analysis becomes particularly relevant when it is desired to obtain an error bound on the model regression when the dimension of the underlying system is high and one needs to work with limited data snapshots, or when the underlying model promotes information redundancies in the observed data stream. For a quick example of the latter scenario, consider the time-series generated by the identity matrix from the initial condition ${\bf 1}$,
% 	identical with those generated by any (row) stochastic matrix.
%		%

	Consider the discrete linear time-invariant system described by the state equation,
	\begin{align}
	\label{discLTI}
	\vec{x}_{t+1} = A \vec{x}_t, \hspace{10mm} t=0,1,2,\ldots,
	\end{align}
	where $A\in \mathbb{R}^{n\times n}$ is the unknown system matrix, $\vec{x}_t\in\mathbb{R}^n$ is the state of the system at time $t$, and the system has been initialized from $\vec{x}_0$.
	%Assume the system starts from an initial point $\vec{x}_0$ that contains all directions in $\mathbb{R}^n$ i.e. $x_0$ can be written as a linear combination of any given basis in $\mathbb{R}^n$.
	The state snapshots collected up to (and including) time $k$ can now be grouped as,
	\begin{equation} 
		\label{dataMatrix}
		X_k=[\vec{x}_0\quad \vec{x}_1 \quad \dots \quad \vec{x}_k]\ , \; Y_k=[\vec{x}_1 \quad \vec{x}_2 \quad \dots \quad \vec{x}_{k+1}].
	\end{equation}
%	Note that $Y_k$ is the time-shifted version of $X_k$ and in fact, $Y_k=AX_k$.
%	From \eqref{discLTI}, it then follows that,\footnote{Same construction is used in the so-called Arnoldi iteration.}
%	\small
%	\begin{align}
%	\label{dataMatrix}
%	X_k=[\vec{x}_0\quad \dots \quad A^k\vec{x}_0]\ , \quad Y_k=[A\vec{x}_0 \quad \dots \quad A^{k+1}\vec{x}_0].
%	\end{align}
%	\normalsize
	This data collection approach is analogous to the first step of the so-called dynamic mode decomposition (DMD) algorithm~\cite{tu2013dynamic}, where this step is followed by the parameterization of the eigenvalues and eigenvectors of the underlying system model.\footnote{The main objective of DMD is however not ``model" regression per se, as it is "modal" fitting, in order to provide useful insights into the underlying (possibly nonlinear) dynamics.}
	Now to estimate the underlying system matrix at each $k$, we consider the the least-squares minimization, $\hat{A}_k = \text{argmin}_A \|Y_k-AX_k\|_F,$ whose solution is of the form $\hat{A}_k=Y_kX_k^{\dagger}$; see Fig.~\ref{ahat}. Note that when ${\textit{\bf rank}}(X_k)=n$, $\hat{A}_k=AX_kX_k^{\dagger}=AX_kX_k^\top(X_kX_k^\top)^{-1}=~A.$
	\begin{figure}[H]
		\begin{center}
			\includegraphics[scale=0.27]{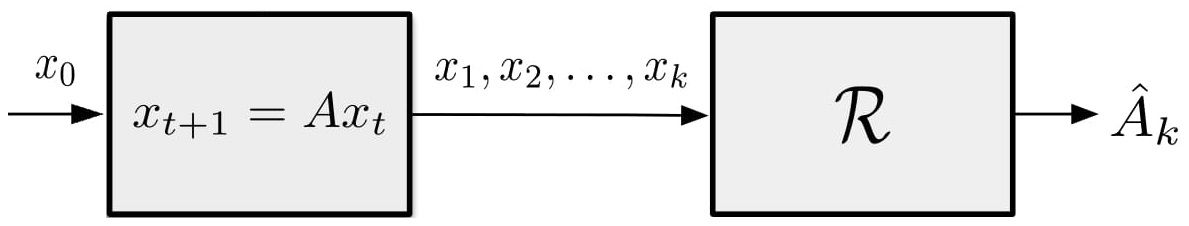}
		\end{center}
		\vspace{-1mm}
		\caption{Estimating the underlying dynamics $A$ after $k$ data snapshots using the model regression ${\cal R}$} \label{ahat}
		\vspace{-2mm}
	\end{figure}
	\noindent The focus of this paper is on the analysis of the error $\|A-\hat{A}_k\|$, i.e., the non-asymptotic error between the original and estimated models, using linear regression when ${\textit{\bf rank}}(X_k) <~n$. 
	%
%	This characterization is particularly useful when it is desired to estimate the underlying linear system that best fits the collected data snapshot without assuming that the initial condition and the model has resulted in a full rank data matrix with the desired dimensions. For example, our work sheds light on the quality of the linear approximation of a nonlinear system by sampling its state evolution around an equilibria.
%	%
%		We also note that such an analysis is pertinent even when the underlying model is in fact linear and it has been assumed that the entire state can be observed noise-free. Removing these assumptions that better capture more practical scenarios however, only add another facet to our analysis despite the fact that noise often removes degeneracies in the data.\footnote{For example, if the state evolution is generated by a (row) stochastic matrix from the initial condition ${\bf 1}$, it is certainly the case that we can perform all necessary matrix inverse operations for estimating the model when the observations of the state are noisy. However, this does not automatically mean that we have a better estimate of the underlying model just because we have noisy observations.}
%		%
%	The goal is to find generic upper bounds on this error and also investigate more concrete results for more specific system structures such as symmetry and multiplicity of eigenvalues.
Our work considers an \textit{online} estimation of the model $A$, where each new data snapshot is added to the previously collected set.
		At any time-step $k$, an estimate for $A$ is found based on the received data up to $k$.
	The resulting data-driven recursive minimization is depicted in Fig.~\ref{ahat}.
	% and Algorithm 1.
	Although not discussed further in this paper, we note that diagram in Fig.~\ref{ahat} can -in principle- be augmented with a model-based control or filtering scheme that utilizes $\hat{A}_k$ after a suitable number of steps.
	%
%	\begin{algorithm}[H]
%		\caption{Estimating the linear dynamic model after $k$ data snapshots}\label{euclid}
%		\begin{algorithmic}[1]
%			\State $t\leftarrow0$
%			\State Given the initial state $\vec{x}_0$ and time-step $k$
%			\State \textbf{For} $t=1,2,\dots,k-1$
%			\State $\quad$ $\vec{x}_{t+1}=A\vec{x}_t$
%			\State \textbf{End for}
%			\State Form $X_k=[\vec{x}_0\ \dots\ \vec{x}_{k-1}]\ , \ Y_k=[\vec{x}_1 \ \dots\ \vec{x}_{k}]$
%			\State Let $\hat{A}_k\in\text{argmin}_A \|Y_k-AX_k\|_F$
%		\end{algorithmic}
%	\end{algorithm} \vspace{-2mm}
	In \S\ref{sec:analysis}, we first introduce an upper bound on the estimation error as a function of the system dynamics $A$, the iteration count $k$, dimension of the system $n$, and the initial conditions $\vec{x}_0$.
In particular, we show that for each time-step, the left-singluar vectors of the SVD of the data matrix dictate the estimation error bound.
	Next, we focus on symmetric system matrices. In this case, it is shown that the model regression error can be characterized by the multiplicities
	in the spectrum of the underlying system.
	%\textcolor{red}{Finally, we redo a similar analysis for the system with control.
	%We will also introduce an optimization problem to find the optimal control at each time-step based on the observed history of the system.}
	
	%=========================================================================
	
	\section{Non-asymptotic Error Analysis}
	\label{sec:analysis}
	
%	\subsection{The General Case}
%	\label{general}
	
	In this section, we examine the error bound for the linear system regression in \eqref{discLTI} based on the system characteristics and the observed data snapshots.
	We assume that $k<n$, i.e., the number of data snapshots is less than the size of the system.
	%\textcolor{red}{It will be shown that for a generic $A$ with simple eigenvalues, $k\geq n$ leads to zero error.}
	The next results, characterizes the regression estimation error as a function of the time step $k$.
	
	\vspace{1mm}
	
	\begin{theorem}
		\label{thm:generalGuarantee}
		Consider the system in \eqref{discLTI} and the corresponding data matrix.
		% in \eqref{dataMatrix}.
		%let $k<n$ and assume data snapshots are linearly independent such that $\text{rank}(X_k)=k$. 
Let $X_k={U}_k \Sigma_k {V}_k^\top$ be the SVD of $X_k$.
		Then the model estimate at time-step $k$ is given by $\hat{A}_k=A(I-E_k)$ where,\footnote{Note that we are quantifying the relative error in (\ref{eq:error}).}
		%\textcolor{red}{(shall we put a remark that says $AE_k=\|\hat{A}_k-A\|$ not $E_k=\|\bar{A}_k-A\|$?)}
		\begin{align}
		\label{eq:error}
		E_k = \left( I - \frac{S_k P_k}{\textit{\bf Tr}{(S_k P_k)}} \right) S_k\,,
		\end{align}
		with,
		\begin{equation} 
		\begin{aligned}
		\label{eq:int_S_Omega}
		S_k = I - {U}_{k-1} {U}_{k-1}^\top, \quad
		P_k = \vec{x}_k\vec{x}_k^\top = A^k \vec{x}_0 \vec{x}_0^\top A^{k^\top}\,.
		\end{aligned}
		\end{equation}
		Moreover,
%		\small
		\begin{align}
		\label{eq:bound1}
		\|E_k\|_2 \leq \left\| I -\frac{S_k P_k}{\textit{\bf Tr}{(S_k P_k)}} \right\|_2 .
		\end{align}
%		\normalsize 
	\end{theorem}
	
%	\vspace{2mm}
	
	\begin{proof}
		%Since $\text{rank}(X_k)=k<n$, 
%We can use the definition for pseudo-inverse $X_k^{\dagger}=(X_k^\topX_k)^{-1}X_k^\top$.
		From \eqref{dataMatrix}, $X_k = \begin{bmatrix} X_{k-1} & A^k \vec{x}_0\end{bmatrix}$, and $Y_k=AX_k = A \begin{bmatrix} X_{k-1} & A^k \vec{x}_0\end{bmatrix}$.
		Then the estimate of the system matrix after the $k$-th snapshot is given by $\hat{A}_k=Y_k X_k^{\dagger}$, where $\hat{A}_k$ is the least-squares solution to $AX_k=Y_k$.
		Thus,
		\begin{align}
		\label{eq:estimate}
		\hat{A}_k = Y_k X_k^{\dagger} = AX_kX_k^{\dagger} = A \begin{bmatrix} X_{k-1} & A^k \vec{x}_0\end{bmatrix} X_k^{\dagger}\ .
		\end{align}
		Hence we need to characterize $X_k^{\dagger}$.
		To this end, we start from $X_k^{\dagger} = \left( X_k^\top X_k \right)^{-1} X_k^\top$.
	We first note that,
%		\small
		\begin{align*}
		X_k^\top X_k &= \begin{bmatrix}
		X_{k-1}^\top \\ \vec{x}_0^\top A^{k^\top}
		\end{bmatrix}
		\begin{bmatrix}
		X_{k-1} & A^k \vec{x}_0
		\end{bmatrix} \\
		&= \begin{bmatrix}
		X_{k-1}^\top X_{k-1} & X_{k-1}^\top A^k \vec{x}_0 \\ \vec{x}_0^\top A^{k^\top} X_{k-1} & \vec{x}_0^\top A^{k^\top} A^k \vec{x}_0
		\end{bmatrix}\,.
		\end{align*}
%		\normalsize
		Then %from \cite{azimi2009adaptable},
		\begin{align*}
		\left( X_k^\top X_k \right)^{-1} = \frac{1}{\zeta} \begin{bmatrix}
		\Phi & -X_{k-1}^{\dagger} A^k \vec{x}_0 
		\\ -\vec{x}_0^\top A^{k^\top} X_{k-1}^{{\dagger}^\top} 
		& 1
		\end{bmatrix},
		\end{align*}
		where
%		\small
		\begin{align*}
		\Phi &= \left(X_{k-1}^\top X_{k-1}\right)^{-1} \left[ \zeta I + X_{k-1}^\top A^k \vec{x}_0 \vec{x}_0^\top A^{k^\top} X_{k-1}^{{\dagger}^\top} \right], \\
		\zeta &= \vec{x}_0^\top A^{k^\top} \left[ I - X_{k-1} \left(X_{k-1}^\top X_{k-1}\right)^{-1} X_{k-1}^\top \right] A^k \vec{x}_0 \\
		&= - \vec{x}_0^\top A^{k^\top} \left[ X_{k-1} X_{k-1}^{\dagger} - I \right] A^k \vec{x}_0\ .
		\end{align*}
%		\normalsize
		In the meantime,
	%	\begin{align*}
		$X_k^{\dagger} = \left( X_k^\top X_k \right)^{-1} X_k^\top \begin{bmatrix} 
		\Psi_{1} \\ \Psi_{2}
		\end{bmatrix}$,
		%\end{align*}
		with,
		\begin{align*}
		\Psi_{1} &= X_{k-1}^{\dagger} + \frac{1}{\zeta} X_{k-1}^{\dagger} A^k \vec{x}_0 \vec{x}_0^\top A^{k^\top} \left({U}_{k-1} {U}_{k-1}^\top - I \right), \\
		\Psi_{2} &= - \frac{1}{\zeta} \vec{x}_0^\top A^{k^\top} \left({U}_{k-1} {U}_{k-1}^\top- I\right),
		\end{align*}
		where we have used the fact $X_{k-1}X_{k-1}^{\dagger}={U}_{k-1}{U}_{k-1}^\top$.
		Thereby, we can expand $\hat{A}_k$ from \eqref{eq:estimate} as,
		\begin{equation} \nonumber
		\begin{aligned}
		&\hat{A}_k = AX_kX_k^{\dagger} = A \left( {U}_{k-1} {U}_{k-1}^\top  \right)\\
		&+ A \left(\frac{1}{\zeta} \left({U}_{k-1} {U}_{k-1}^\top- I\right) A^k \vec{x}_0 \vec{x}_0^\top A^{k^\top} \left({U}_{k-1} {U}_{k-1}^\top - I \right)\right)\\
		&= A \left( {U}_{k-1} {U}_{k-1}^\top   \right)\\
		&- A \left( \frac{ \left({U}_{k-1} {U}_{k-1}^\top - I\right) A^k \vec{x}_0 \vec{x}_0^\top  A^{k^\top} \left({U}_{k-1} \mathcal{U}_{k-1}^\top - I \right)}{\vec{x}_0^\top A^{k^\top} \left({U}_{k-1} {U}_{k-1}^\top - I \right) A^k \vec{x}_0 } \right)\, ,
		\end{aligned} %\label{A_DMD}
		\end{equation}
		and from \eqref{eq:int_S_Omega}, the estimated model $\hat{A}_k$ at time-step $k$ is given by $\hat{A}_k = A \left( I - E_k \right)$ with,
%		\small
		\begin{align*}
		E_k &= \frac{\Big(\textit{\bf Tr}{(S_k P_k)} I -S_k P_k \Big) S_k}{\textit{\bf Tr}{(S_k P_k)}} = \left( I - \frac{S_k P_k}{\textit{\bf Tr}{(S_k P_k)}} \right) S_k.
		\end{align*} 
%		\normalsize
		The magnitude of this error simplifies for the case of the spectral norm as,
%		\small
		\begin{align*}
		\|E_k\|_2 = \left\| \left( I - \frac{S_k P_k }{\textit{\bf Tr}{(S_k P_k)}} \right) S_k \right\|_2 \leq \left\| I -\frac{S_k P_k}{\textit{\bf Tr}{(S_k P_k)}} \right\|_2 ,
		\end{align*}
%		\normalsize 
		since $\|S_k\|_2=1$ for $k<n$.
	Lastly, we note that $S_k$ is the projection onto the null space of $X_{k-1}$ and $P_k$ is the covariance matrix of the data at time-step $k$; as such both matrices are positive-semidefinite.		
		\end{proof}
		%Moreover since $\mathcal{U}_{k-1}$ is unitary, $S_k$ has only zero and one eigenvalues.
	\vspace{1mm}	

	We note that the relation (\ref{eq:error}) captures -in a succinct way- the dependency of the model regression error on how new modes are revealed by the data stream over time.
	%\vspace{2mm}
	%
	%\begin{remark}
	%	Based on the definitions in \eqref{eq:int_S_Omega}, inequality \eqref{eq:bound1} finds an upper bound on the estimate error as a function of system and data parameters $A$, $\vec{x}_0$, $\mathcal{U}_{k-1}$, and the number of time-steps $k$.
	%\end{remark}
	%\vspace{2mm}
%	\subsection{Regression Error Ordering for PSD Models}

%	Before we examine the role of eigenvalue multiplicities in dictating the model regression error for
%	symmetric systems, we are in a position to show that for positive semi-definite system matrices,
%	the regression error satisfies a (positive semi-definite) ordering property; see the Appendix for the proof of this result.
%	%	We can further strengthen the above observation for positive semi-definite system matrices as follows.
%	%	next show that the estimated matrix $\hat{A}$ obtained from \eqref{eq:bound1} is bounded.
%		\vspace{2mm}
%	
%	\begin{lemma}
%		\label{lem:BoundS_Omega}
%		For all initial conditions $\vec{x}_0$ and the system matrix
%		 $A = A^\top \succeq 0$, we have $0 \preceq \hat{A} \preceq A$.
%		
%	\end{lemma}
	
	\subsection{Non-asymptotic Error Analysis for Symmetric Systems}
	
	\label{section:Symmetry}
	
	In this section, we consider linear systems with symmetric dynamics with the aim of characterizing fundamental bounds on the regression error in terms of the spectral properties of the system.
	This insight into the regression error is achieved through the spectral decomposition of the
	system matrix,
		\begin{align}
	\label{eq:symmetry}
	A=Q\Lambda Q^\top=\sum_{i=1}^r \lambda_i\vec{q}_i\vec{q}_i^\top,
\end{align}
	where $Q$ is the unitary matrix containing the eigenvectors corresponding to nonzero eigenvalues of $A$, $\Lambda$ is the diagonal matrix of nonzero eigenvalues, and $r=\textbf{rank}(A)$.
	Symmetric system matrices appear in a wide range of applications where interactions leading to the dynamics is bidirectional; such systems are of interest in biological networks~\cite{yeung2002reverse}, social interactions~\cite{alemzadeh2017controllability}, robotic swarms~\cite{bullo2009distributed}, and networks security~\cite{alaeddini2017adaptive}.
Using this spectral decomposition of symmetric systems, we show
that the regression error is dependent on the multiplicity of eigenvalues in $A$.
In particular, we show that if $m(\lambda)=1, \forall \lambda\in\mathbf{\Lambda}(A)$, then the
upper bound \eqref{eq:bound1} is a function of the largest and smallest singular values of the
system matrix as well as its rank. Otherwise, the regression error is 
$\max_{i:m(\lambda_i)>1} |\lambda_i|$.
We provide the details of the approach for each case.
%	\vspace{2mm}
	
	\subsubsection{Simple Eigenvalues}
	
	\label{section:allSimple}
	We first consider the case when the symmetric system matrix has simple eigenvalues and rank $r$.
	In reference to (\ref{eq:symmetry}), consider the entire set of eigenvectors of $A$ consisting of $Q=[ \vec{q}_1\ \cdots\ \vec{q}_r ]$ and $\bar{Q}=[ \bar{\vec{q}}_{r+1}\ \cdots\ \bar{\vec{q}}_n ]$, where $\{\vec{q}_1,\cdots, \vec{q}_r, \bar{\vec{q}}_{r+1},\cdots,\bar{\vec{q}}_n \}$ forms a basis for the entire $\mathbb{R}^n$.
	Then the nonzero random initial state $\vec{x}_0$ can be written as,
	\begin{align}	\label{eq:IC}
	\vec{x}_0 = Q\nu + \bar{Q}\mu = \sum_{i=1}^{r} \nu_i \vec{q}_i + \sum_{i=r+1}^{n} \mu_i \bar{\vec{q}}_i\quad \quad \vec{\nu} \neq 0.
	\end{align}
	%	\vspace{2mm}
	\begin{lemma}
		\label{lem:diff}
		For the symmetric linear system decomposed as~\eqref{eq:symmetry}, we have,
%		\begin{align*}
		$A-\hat{A}_k = A \Big( I - \frac{S_k Q \Lambda^k \nu \nu^\top  \Lambda^k Q^\top}{ \| S_k Q \Lambda^k \nu\|^2 } \Big) S_k$.
%		\end{align*}
		
	\end{lemma}
	
	\begin{proof}
		From \eqref{eq:IC} and \eqref{eq:int_S_Omega} we have, $P_k = A^k \vec{x}_0 \vec{x}_0^\top A^{k^\top}$
		%\scriptsize
		\begin{align*}
		&= (Q\Lambda^kQ^\top)(Q\nu + \bar{Q}\mu)(Q\nu + \bar{Q}\mu)^\top(Q\Lambda^kQ^\top) \\
		&= Q \Lambda^k \nu \nu^\top \Lambda^k Q^\top, 
		\end{align*}
		%\normalsize
		and,
		%\scriptsize
		%\begin{align*}
		$\textit{\bf Tr} (S_kP_k) = \textit{\bf Tr} (S_kQ \Lambda^k \nu \nu^\top \Lambda^k Q^\top) = \|S_kQ\Lambda^k\nu\|^2$,
		%\\
%		&= \textit{\bf Tr} (\nu^\top \Lambda^k Q^\top S_k Q \Lambda^k \nu) \\
%		&= \nu^\top \Lambda^k Q^\top S_k^2 Q \Lambda^k \nu = (S_kQ\Lambda^k\nu)^\top (S_kQ\Lambda^k\nu) \\
		%= \|S_kQ\Lambda^k\nu\|^2$, 
		%\end{align*}
		%\normalsize
		where we have used that $S_k$ is a symmetric projection.
		Substituting these in \eqref{eq:error}
		%\small
%		\begin{align*}
%%		A-\hat{A}_k = A E_k &= A \Big( I - \frac{S_kP_k}{\textit{\bf Tr}{(S_kP_k)}} \Big) S_k \\
%		= A \Big( I - \frac{S_k Q \Lambda^k \nu \nu^\top \Lambda^k Q^\top }{ \| S_k Q \Lambda^k \nu\|^2 } \Big) S_k,
%		\end{align*}
		completes the proof.
		%\normalsize
	\end{proof}
	%\begin{equation} 
	%\begin{aligned}
	%&X_{k-1} X_{k-1}^T S_k = S_k X_{k-1} X_{k-1}^T = 0\\
	%&P_k = Q \Lambda^k \nu \nu^T \Lambda^k Q^T \\
	%&S_k P_k S_k= S_k Q \Lambda^k \nu \nu^T \Lambda^k Q^T S_k\\
	%&\trace(S_k P_k) = \| S_k Q \Lambda^k \nu\|^2\,.
	%\end{aligned}
	%\end{equation}

	\vspace{1mm}	

	We now show that when $k<n$, the error depends on the largest and smallest eigenvalues of $A$.
	
%	\vspace{1calmm}
	
	\begin{theorem}
		\label{thm:singularValue}
		Consider the linear dynamical system with symmetric system matrix $A$ as in \eqref{eq:symmetry} and 
		the initial state $x_0$ as in \eqref{eq:IC}. If $k<n$, then
		\begin{align*}
		\|A-\hat{A}_k\|_F^2 \leq \Big( n-\min \{ k, |\nu|+\min\{ |\mu|,1 \} \} \Big) \lambda_1^2 - \lambda_n^2\,,
		\end{align*}
		where $\lambda_n = \lambda_{\min}(A),$ $\lambda_1 = \lambda_{\max}(A),$ and $|\nu|$ and $|\mu|$ are the number of nonzero $\nu_i$'s and $\mu_i$'s, respectively.
	\end{theorem}
	
\begin{proof}
		From Lemma~\ref{lem:diff} we observe that,
		%\footnotesize
		\begin{equation*}
		\begin{aligned}
		&\|A-\hat{A}_k\|_F^2 = \textit{\bf Tr} \Big( (A-\hat{A}_k)^\top (A-\hat{A}_k) \Big) \\
		&= \textit{\bf Tr} (\Lambda^2 Q^\top S_k Q) - \frac{\nu^\top \Lambda^k Q^\top S_k Q \Lambda^2 Q^\top S_k Q \Lambda^k \nu }{\nu^\top \Lambda^k Q^\top  S_k Q \Lambda^k \nu} \\
		& = \| A S_k\|^2- \frac{\|A S_k Q \Lambda^k \nu \|^2}{\| S_k Q \Lambda^k \nu\|^2}\,.
		\end{aligned}
		\end{equation*}
		%\normalsize
		In the meantime, 
		\begin{equation*} 
		\begin{aligned}
		\| A S_k\|_F^2 &\leq \text{\bf rank}(S_k) \|AS_k\|^2_2  \\
		&\leq \text{\bf rank}(S_k) \|A\|_2^2\|S_k\|_2^2 = \text{\bf rank}(S_k) \lambda_{1}^2(A);
		\end{aligned}
		\end{equation*}
		moreover, since
%		\begin{equation*} 
		$\lambda_{n}(A)=\inf\limits_{y\neq0} {\|Ay\|_2}/{\|y\|_2}$, we have 
	%	\end{equation*}
		\begin{equation*}
		\frac{\|A S_k Q \Lambda^k \nu \|^2}{\| S_k Q \Lambda^k \nu\|^2} \geq  \lambda_{n}^2(A) .
		\end{equation*}
		Since $A\vec{q}_i=\lambda_i \vec{q}_i$ and $A\bar{\vec{q}}_i=0$, we have,
		\begin{align*}
		&X_{k-1} = [\vec{x}_0 \quad \vec{x}_1 \quad \cdots \quad \vec{x}_{k-1}] \\
		&=\begin{bmatrix} \sum\limits_{i=1}^{r} \nu_i \vec{q}_i + \sum\limits_{i=r+1}^{n} \mu_i \bar{\vec{q}}_i & \sum\limits_{i=1}^{r} \lambda_i \nu_i \vec{q}_i & \cdots & \sum\limits_{i=1}^{r} \lambda_i^{k-1} \nu_i \vec{q}_i \end{bmatrix}\,.
		\end{align*}
		Thus,
		\begin{equation} 
		\text{\bf rank}(X_{k-1}) = \min \{ k, |\nu|+\min\{ |\mu|,1 \} \}\,,
		\end{equation}
		and,
		$$ \text{\bf rank}(S_k) = n - \text{\bf rank}(X_{k-1}) = n-\min \{ k, |\nu|+\min\{ |\mu|,1 \} \} \,.$$
		Hence,
		\begin{equation} \nonumber
		\begin{aligned}
		\|A-\hat{A}_k\|_F^2 \leq \Big( n-\min \{ k, |\nu|+\min\{ |\mu|,1 \} \} \Big) \lambda_1^2 - \lambda_n^2,
		\end{aligned}
		\end{equation}
		which completes the proof.
	\end{proof}	
	%\vspace{2mm}
	%\begin{remark}
	%	The result in \cref{thm:singularValue} shows how the span of eigenvalues of the dynamics $A$ can affect the error bound.
	%	Also notice the dependency of the bound on parameters $|\nu|$ and $|\mu|$ that originate from the random vector $\vec{x}_0$. 
	%\end{remark}
	
	%%%%%%%%%%%%%%%%%%%%%%%%%%%%%%%%%%%%%%%%%%%%%%%%
	
	\subsection{Effect of Eigenvalues Multiplicity on the Regression Error}
	
	\label{sec:multi}
	
	In this section, we consider the symmetric systems whose eigenvalues are not necessary simple.
	We will see that for such systems, the regression error $\|E_k\|$ converges to the largest eigenvalue with multiplicity greater than one, i.e., $\|E_k\|=\max_{i:m(\lambda_i)>1} |\lambda_i|$ for $k\geq n$.
	
 In order to show this, we will pursue the convention adopted in \eqref{dataMatrix}, \eqref{eq:symmetry}, and \eqref{eq:IC}.
	As in \S~\ref{section:Symmetry}, let $r=\text{rank}(A)$ and define $\tilde{Q}=[Q\ |\ \bar{Q}]=[\vec{q}_1 \ \dots \ \vec{q}_r \ \ \bar{\vec{q}}_{r+1} \ \dots \ \bar{\vec{q}}_{n}]$, where the columns of $\tilde{Q}$ span the entire $\mathbb{R}^n$.
	Furthermore, let $\alpha = [\nu^\top \ \mu^\top]^\top$, where $\alpha$ and $\mu$ are from \eqref{eq:IC}.
	The data matrix can now be re-written as,
	\begin{equation}
	\begin{aligned}
	\label{eq:dataNew}
	X_k &= [\vec{x}_0\quad A\vec{x}_0\quad A^2\vec{x}_0\quad \dots \quad A^k\vec{x}_0] \\
%	&= [\tilde{Q}\tilde{Q}^\top \vec{x}_0\quad \tilde{Q}\Lambda \tilde{Q}^\top \vec{x}_0\quad \tilde{Q}\Lambda^2 \tilde{Q}^\top \vec{x}_0\ \dots \ \tilde{Q}\Lambda^k \tilde{Q}^\top \vec{x}_0] \\
	&= \tilde{Q}\ [\tilde{Q}^\top \vec{x}_0\quad \Lambda \tilde{Q}^\top \vec{x}_0\quad \dots \quad \Lambda^k \tilde{Q}^\top \vec{x}_0]. %\quad \Lambda^2 \tilde{Q}^\top\vec{x}_0
	\end{aligned}
	\end{equation}
	From \eqref{eq:IC} and the fact that the columns of $\tilde{Q}$ are orthonormal, we have $\tilde{Q}^\top \vec{x}_0=[\alpha_1\ \alpha_2\ \dots\ \alpha_n]^\top $ and $\Lambda^j\tilde{Q}^\top \vec{x}_0=[\alpha_1\lambda_1^j\quad \alpha_2\lambda_2^j\quad \dots\quad \alpha_n\lambda_n^j]^\top$.
	Moreover, in light of \eqref{eq:dataNew} we can decompose the data matrix into $X_k=\tilde{Q}\Gamma V$
%	\begin{align}
%	\label{eq:decomposition}
%	X_k=\tilde{Q}\Gamma V,
%	\end{align}
	where,
%	\small
	\begin{align}
	\label{eq:dataDecomposition}
	\Gamma = \begin{bmatrix}
	\alpha_1 & 0 & \dots & 0 \\
	0 & \alpha_2 & \dots & 0 \\
	\vdots & \vdots & \ddots & \vdots \\
	0 & 0 & \dots & \alpha_n
	\end{bmatrix},\quad
	V = \begin{bmatrix}
	1 & \lambda_1 & \dots  & \lambda_1^k \\
	1 & \lambda_2 & \dots  & \lambda_2^k \\
	\vdots   & \vdots      & \ddots & \vdots \\
	1 & \lambda_n & \dots & \lambda_n^k
	\end{bmatrix}.
	\end{align}
%	\normalsize
	The matrix $[V]_{ij}=\lambda_i^{j-1}$ is the Vandermonde matrix formed by the eigenvalues of $A$ and $\Gamma=\text{\bf Diag}([\alpha_i]_{i=1}^n)$.\footnote{Note that it is assumed that for all $i$, $\vec{x}_0\not\perp \vec{q}_i$; in this case, $\alpha_i\neq 0$ for all $i$ and $\textbf{rank}(X_k)=\textbf{rank}(V)$.}
	Assume now that the system matrix $A$ contains $s$ distinct eigenvalues and let $\mathbf{\Lambda}^*(A)=\{ \lambda_{t_1}, \lambda_{t_2},\dots, \lambda_{t_s} \}$ be the set of these eigenvalues; as such, the other $n-s$ eigenvalues are repetitions of the elements in $\mathbf{\Lambda}^*(A)$.
	
	\begin{comment}
	
	\vspace{2mm}
	
	\begin{remark}
	\label{rmk:rankOfA}
	The assumption of $\text{rank}(A)=n$ is not necessary and we can pursue with the setup in \cref{section:allSimple}.
	The only requirement is that $\Gamma$ has full-rank which means that the initial condition contains every direction in the space.
	However, without loss of generality, we assume $A$ has full-rank since the formulations get simplified with $\mu=0$ in \eqref{eq:IC}.
	\end{remark}
	
	\vspace{2mm}
	
	\end{comment}
	
	For our subsequent error analysis, we will use the rank of $X_k$ at each time-step.
	The next result characterizes the rank of $X_k$  based on the number of the collected data.
%	First, an intermediate observation on the rank of the Vandermonde matrix.
	
	\vspace{1mm}
	
	\begin{lemma}
		\label{lem:vender}
		Given $k<s$, the $s\times k$ Vandermonde matrix defined by $[V_s]_{ij}=\lambda_i^{j-1},\ i\in\{1,\dots,s\},\ j\in\{1,\dots,k\}$, formed by the elements of $\mathbf{\Lambda}^*(A)$, has full-rank.
	\end{lemma}

		\begin{proof}
		%Define $V_r$ as follows
		%	\begin{align*}
		%	V_r=\begin{bmatrix}
		%	1 & \lambda_1 & \dots  & \lambda_1^{k-1} \\
		%	1 & \lambda_2 & \dots  & \lambda_2^{k-1} \\
		%	\vdots   & \vdots      & \ddots & \vdots \\
		%	1 & \lambda_r & \dots & \lambda_r^{k-1}
		%	\end{bmatrix}
		%	\end{align*}
		Let $v_i$ be the $i$th column of $V_s$ and assume that $c_1v_1+c_2v_2+\dots+c_kv_k=0$.
		Consider row $p$ of the equation $c_1+c_2\lambda_p+\dots+c_k\lambda_p^k=0$.
		Since $\lambda_i\neq\lambda_j$ for $i\neq j$, there exist $s$ solutions to the $k$-degree polynomial, $$P(x)=c_0+c_1x+\dots+c_kx^k=0.$$
		Hence $c_1=c_2=\dots=c_k=0$ and $v_i$'s are linearly independent and since $k<s$, $V_s$ has full-rank.
	\end{proof}
	
	\vspace{1mm}
	
	\begin{theorem}
		\label{thm:rank}
		Let $k$ be the number of collected data snapshots and $s=|\mathbf{\Lambda}^*(A)|$.
		Then $\text{\bf rank}(X_k) = k$ when $k < s$ and $\text{\bf rank}(X_k) = s$ when $k \geq s$.
%		\begin{align*}
%		\text{\bf rank}(X_k) = \begin{cases}
%		k \hspace{5mm} \text{if}\ \ k<s \\
%		s \hspace{5.3mm} \text{if}\ \ k\geq s \\
%		\end{cases}
%		\end{align*}
	\end{theorem}

	\begin{proof}
		For $k<s$, it is straightforward to show that from Lemma~\ref{lem:vender}, $\text{\bf rank}(X_k)=\text{\bf rank}(V)=\text{\bf rank}(V_s)=k$. For $k\geq s$, we know that $\text{\bf rank}(V)=\text{\bf rank}(V_s)$, where $[V_s]_{ij}=\lambda_i^{j-1},\ i\in\{1,\dots,s\},\ j\in\{1,\dots,k+1\}$. Since $V_s$ has full-rank (this can be shown using the nonzero sub-matrix determinants for the first $s\times s$ block), we get $\text{\bf rank}(X_k)=\text{\bf rank}(V)=\text{\bf rank}(V_s)=s$.
	\end{proof}

%Please refer to the Appendix for the proofs of the above two observations.

	Let $k\geq s$ and $X_k={U}\Sigma V^\top$ be the SVD of $X_k$ where,
%	\footnotesize
	\begin{align}
	\label{eq:SVD}
	X_k=[{U}_1\quad {U}_2]\begin{bmatrix}
	\Sigma_1 & \vec{0}_{s\times (k-s)} \\ \vec{0}_{(n-s)\times s} & \vec{0}_{(n-s)\times (k-s)}
	\end{bmatrix}
	\begin{bmatrix}
	{V}_1 & {V}_2
	\end{bmatrix}^\top,
	\end{align}
%	\normalsize
	with ${U}_1\in\mathbb{R}^{n\times s},\ {U}_2\in\mathbb{R}^{n\times (n-s)},\ \Sigma_1\in\mathbb{R}^{s\times s},\ {V}_1\in\mathbb{R}^{k\times s},\ {V}_2\in\mathbb{R}^{k\times (k-s)}$.
	The existence of ${U}_2$ and ${V}_2$ is due to the fact that $X_k$ is degenerate.
	Without loss of generality, we re-arrange the columns of $\tilde{Q}$ and $\Lambda$ such that,
	\begin{align}
	\label{eq:wlog}
	\Lambda_{R}=\left[
	\begin{array}{c|c}
	\Lambda_1 & \vec{0} \\
	\hline
	\vec{0} & \Lambda_2
	\end{array}
	\right]
	\quad , \quad
	\tilde{Q}_R=\left[
	\begin{array}{c|c}
	\tilde{Q}_1 & \tilde{Q}_2
	\end{array}
	\right],
	\end{align}
	where $\Lambda_1=\text{\bf Diag}([\lambda_{t_1}\quad \lambda_{t_2}\quad \dots\quad \lambda_{t_s}])$ contains the element of $\mathbf{\Lambda}^*(A)$ with the corresponding eigenvectors in $\tilde{Q}_1\in\mathbb{R}^{n\times s}$.
	The remaining eigenvalues and the corresponding eigenvectors are stacked in $\Lambda_2\in\mathbb{R}^{(n-s)\times (n-s)}$ and $\tilde{Q}_2\in\mathbb{R}^{n\times (n-s)}$, respectively.
	
	A crucial term in our analysis is ${U}_2^\top\tilde{Q}_R\in\mathbb{R}^{(n-s)\times n}$.
	This matrix multiplication combines a submatrix of ${U}$ corresponding to the repeated eigenvalues, with the orthonormal eigenvectors of the system.
	In the next result, we show that this term has a specific row structure.
	
	\vspace{1mm}
	
	\begin{lemma}
		\label{lem:rowStructure}
		Given the convention in \eqref{eq:wlog}, we have ${U}_2^\top \tilde{Q}_1=0$, i.e., ${U}_2^\top\tilde{Q}_R = {U}_2^\top \big[ \tilde{Q}_1 \mid \tilde{Q}_2 \big] = \big[ \ \vec{0}_{(n-s)\times s} \mid {U}_2^\top \tilde{Q}_2 \ \big]$.
	\end{lemma}
	
	%\vspace{2mm}
	
	\begin{proof}
		From \eqref{eq:SVD} we have ${U}_2^\top X_k=0$.
		Then,
%		\footnotesize
		\begin{align*}
		{U}_2^\top X_k &= {U}_2^\top \tilde{Q}\tilde{Q}^\top X_k = {U}_2^\top \tilde{Q}\tilde{Q}^\top [x_0\quad Ax_0\quad \dots\quad A^kx_0] \\
		&= {U}_2^\top \tilde{Q} \big[ \tilde{Q}^\top x_0\quad \Lambda \tilde{Q}^\top x_0\quad \dots\quad \Lambda^k \tilde{Q}^\top x_0 \big] \\
		&= {U}_2^\top \tilde{Q} \Gamma V = 0,
		\end{align*}
%		\normalsize
		with $\Gamma$ and $V$ defined as in \eqref{eq:dataDecomposition}.
		Define $B={U}_2^\top \tilde{Q}$ and let $\vec{b}_i=[\vec{b}_{i_1}\quad \vec{b}_{i_2}\quad \dots \quad \vec{b}_{i_n}]$ and $\vec{v}_i=[1\quad \lambda_i\quad \dots\quad \lambda_i^k]$ be the $i$'th rows of $B$ and $V$, respectively.
		Then for all $i\in\{1,2,\dots,n-s\}$ we have
	%	\begin{align*}
		$B\Gamma V=0 $ implying that $\vec{b}_i\Gamma V=0$; as such, $\sum_{j=1}^n \vec{b}_{i_j} \alpha_j \vec{v}_j = 0$ implies that $\sum_{t=1}^s \vec{c}_{t}\vec{v}_t^* = 0$,
		where $\vec{v}_t^*$'s are the rows of $V$ corresponding to $s$ distinct eigenvalues and $\vec{c}_t$'s are some combinations of $\vec{b}_{i_j}\alpha_j$'s.
		Since the vectors $\vec{v}_t^*$'s are linearly independent, we get $\vec{c}_i=0$ for all $i=1,2,\dots,s$.
		Considering that $\vec{c}_i=\vec{b}_{i_j}$ for the elements with simple eigenvalues, the result implies that for any row $\vec{v}_j$ corresponding to a simple eigenvalue $\lambda_j$ the corresponding coefficient $\vec{b}_{i_j}=0$.
		Hence, the structure $\big[ \ \textbf{0}_{(n-s)\times s} \mid {U}_2^\top \tilde{Q}_2 \ \big]$ follows, implying that ${U}_2^\top \tilde{Q}_1=0$.
	\end{proof}
	
	\vspace{1mm}
	
	\begin{definition}
		\label{def:U2}
		Define ${U}_2'$ by permuting the columns of ${U}_2$ such that we obtain block diagonal matrix ${U}_2'^\top \tilde{Q}_2=\text{\bf Diag}([P_1,P_2,\dots,P_{\ell}])$, where $\ell$ is the number of eigenvalues $\lambda$ with $m(\lambda)>1$ and $P_i={U}_2^{'i^\top}\tilde{Q}_2^i\in\mathbb{R}^{(m(\lambda_i)-1)\times m(\lambda_i)}$; as such ${U}_2^{'i}\in\mathbb{R}^{n\times (m(\lambda_i)-1)}$ is the matrix containing vectors in ${U}_2'$ corresponding to $\lambda_i$ and $\tilde{Q}_2^i\in\mathbb{R}^{n\times m(\lambda_i)}$ is the matrix of eigenvectors corresponding to $\lambda_i$.
	\end{definition}
	
	\vspace{1mm}
	
	\begin{remark}
		To justify the existence of such a matrix ${U}_2'$, notice that the SVD factorization in terms of ${U}$ and ${V}$ are not unique and for any such factorization, ${U}_1\perp {U}_2$.
	\end{remark}
	
	\vspace{1mm}
	
	We are now well positioned to prove the main theorem of this section.
	
	\vspace{1mm}
	
	\begin{theorem}
		Consider the dynamics represented by \eqref{discLTI} and \eqref{dataMatrix}.
		Let $s=|\mathbf{\Lambda}^*(A)|$ be the number of distinct eigenvalues of $A$.
		Assume that $k\geq s$ and let $\lambda^*=\max_{i:m(\lambda_i)>1} |\lambda_i|$ be the largest eigenvalue of $A$ with multiplicity greater than one.
		Then $\|E_k\|_2=\|A-\hat{A}_k\|_2=\lambda^*$.
	\end{theorem}
	
%	\vspace{2mm}
	
	\begin{proof}
		We will show that $\mathbf{\Lambda}(E_k)=\mathbf{\Lambda}(A)\backslash\mathbf{\Lambda}^*(A)$.
		The error can then be re-written as,
%		\small
		\begin{align*}
		\|&E_k\|_2 \\
		&= \|A-\hat{A}_k\|_2
		= \|A-AX_kX_k^{\dagger}\|_2
		= \|A(I-X_kX_k^{\dagger})\|_2 \\
		&= \left\Vert  A \Big( I-{U}_1\Sigma_1 {V}_1^\top{V}_1\Sigma_1^{-1}{U}_1^\top \Big) \right\Vert_2
		= \|A(I-{U}_1{U}_1^\top)\|_2 \\
		&= \|A{U}_2'{U}_2'^\top\|_2
		= \|\tilde{Q}\Lambda\tilde{Q}^\top{U}_2'{U}_2'^\top\|_2
		= \|\Lambda\tilde{Q}^\top U_2'{U}_2'^\top\tilde{Q}\|_2 \\
		&= \|\Lambda ({U}_2'^\top \tilde{Q})^\top ({U}_2'^\top \tilde{Q})\|_2,
		\end{align*}
%		\normalsize
		where ${U}_2'$ pertains to \cref{def:U2} and $I-{U}_1{U}_1^\top={U}_2'{U}_2'^\top$ is the projection matrix onto ${\cal N}(X_k^\top)$.
		Note that since the columns of ${U}_2'$ are linearly independent, we have $\text{\bf rank}(E_k)=\text{\bf rank}(A{U}_2'{U}_2'^\top)=n-s$.
		From Lemma~\ref{lem:rowStructure},
		\begin{equation}
		\begin{aligned}
		\label{eq:structure}
		&\Lambda({U}_2'^\top \tilde{Q})^\top ({U}_2'^\top \tilde{Q}) = \Lambda\begin{bmatrix}
		\vec{0} \\ ({U}_2'^\top \tilde{Q}_2)^\top 
		\end{bmatrix}
		\begin{bmatrix}
		\vec{0} & {U}_2'^\top\tilde{Q}_2
		\end{bmatrix} \\
		&= \left[
		\begin{array}{c|c}
		\vec{0}_{s\times s} & \vec{0}_{s\times(n-s)} \\
		\hline
		\vec{0}_{(n-s)\times s} & \Lambda_2({U}_2'^\top\tilde{Q}_2)^\top({U}_2'^\top \tilde{Q}_2)
		\end{array}
		\right].
		\end{aligned}
		\end{equation}
		Then from \cref{def:U2}, $\Lambda_2({U}_2'^\top\tilde{Q}_2)^\top ({U}_2'^\top\tilde{Q}_2)=\text{\bf Diag}([\lambda_1P_1^\top P_1\ ,\ \lambda_2P_2^\top P_2\ ,\ \dots\ ,\ \lambda_{\ell}P_{\ell}^\top P_{\ell}])$.
		Consider the $i$th block $\lambda_iP_i^\top P_i$. Notice that from definition $\text{\bf rank}(P_i^\top P_i)=m(\lambda_i)-1$, and since ${U}_2^{'i^\top}$ and $\tilde{Q}_2^i$ are orthonormal, $\lambda_iP_i^\top P_i=\lambda_i\tilde{Q}_2^{i^\top}{U}_2^{'i}{U}_2^{'i^\top}\tilde{Q}_2^i$ has the spectrum $\mathbf{\Lambda}(\lambda_iP_i^\top P_i)=\{0,\lambda_i,\dots,\lambda_i\}$.
		Then having $\ell$ of these blocks $\mathbf{\Lambda}(E_k)=\mathbf{\Lambda}(A)\backslash\mathbf{\Lambda}^*(A)$, i.e., the spectrum of $E_k$ contains all repeated eigenvalues of $A$.
		
		Since both $\Lambda$ and $({U}_2'^\top \tilde{Q})^\top ({U}_2'^\top \tilde{Q})$ are symmetric square block diagonal matrices, the product $\Lambda({U}_2'^\top \tilde{Q})^\top ({U}_2'^\top \tilde{Q})$ is symmetric and therefore
		%\begin{align*}
		$\|E_k\|_2 = \|\Lambda({U}_2'^\top \tilde{Q})^\top ({U}_2'^\top \tilde{Q})\|_2 = \max\limits_{i=1,\dots,\ell}\lambda_i = \lambda^*$.\footnote{It is straightforward to show an analogous result for the Frobenius norm $\|E_k\|_F^2=\|A-\hat{A}_k\|_F^2=\sum_{i=1}^{\ell} [m(\lambda_i)-1] \lambda_i^2$.}
	\end{proof}

	\vspace{-2mm}
	
\section{Model Regression on Networks}
	\label{sec:sims}
%\subsection{Model Regression on Networks}
	We now provide an example to demonstrate the applicability of the error bounds 
	on networked systems.
	%
%	In particular, we proceed to apply these error bounds to  two types of networked systems: the \emph{Petersen} graph and \emph{circulant} graphs.
%	\subsection{Regression on Petersen Networks}
	%Petersen graph is a well-studied construct in graph theory.
	%
		Consider the Petersen graph on 10 nodes and 15 edges as shown in Fig. \ref{fig:plotPetersen}.
	We use the weighted version of this specific structure to find error bounds on a system with simple eigenvalues.
%	\begin{figure}[h]
%		\centering
%		\begin{minipage}{0.45\linewidth}
%			\centering
%			\subfloat[]{\label{fig:petersen}\includegraphics[scale = 0.1]{petersen}}
%		\end{minipage}
%		\hspace{1mm}
%		\begin{minipage}{0.45\linewidth}
%			\centering
%			\subfloat[]{\label{fig:circulant}\includegraphics[scale = 0.25]{circulant}}
%		\end{minipage}
%		\caption{(a) Weighted Petersen Graph (b) Circulant Graph $C_{10}(\{1,2,3\})$}
%		\label{fig:systems}
%	\end{figure}
	The dynamics of this system is defined using the \emph{graph Laplacian}, defined as $\mathcal{L}=D-A$, where $A$ is the \emph{adjacency} matrix that defines the connections in the network and $D$ is the \emph{degree} matrix defined as $D_{ii}=\sum_j|W_{ij}|$; in this case $W_{ij}$ is the weight of the edge between nodes $i$ and $j$.
	Network symmetries typically induce eigenvalue multiplicities in the corresponding adjacency and Laplacian matrices.
	Hence to make the system more generic, we add weights $w_{1,6}=1$, $w_{2,7}=2$, $w_{3,8}=3$, $w_{4,9}=4$, and $w_{5,10}=5$ and for all other weights we have $w_{i,j}=1$.
	For each component $i$ the dynamics depend on the adjacent nodes in the graph $\dot{x}_i = \sum_j |W_{ij}|(x_i-x_j)$.
	Then the overall dynamics can be written as $\dot{x}=-\mathcal{L}x$.
	The model regression algorithm discussed in this paper leads to the error shown in Fig. \ref{fig:plotPetersen}.
	For this simulation, the initial condition has been chosen as a (normalized) random vector $x_0\in\mathbb{R}^{10}$.
	The upper subfigure shows the comparison of the bound for general case using the spectral norm and the lower subfigure demonstrates the same setup for \ref{section:Symmetry}.
	It can be seen that the error converges to zero after $k=10$ steps, since the system matrix
	has simple eigenvalues.
	\begin{figure}[h]
	%	\centering
	\vspace{-.05in}
		\raisebox{12mm}{\includegraphics[scale=.07]{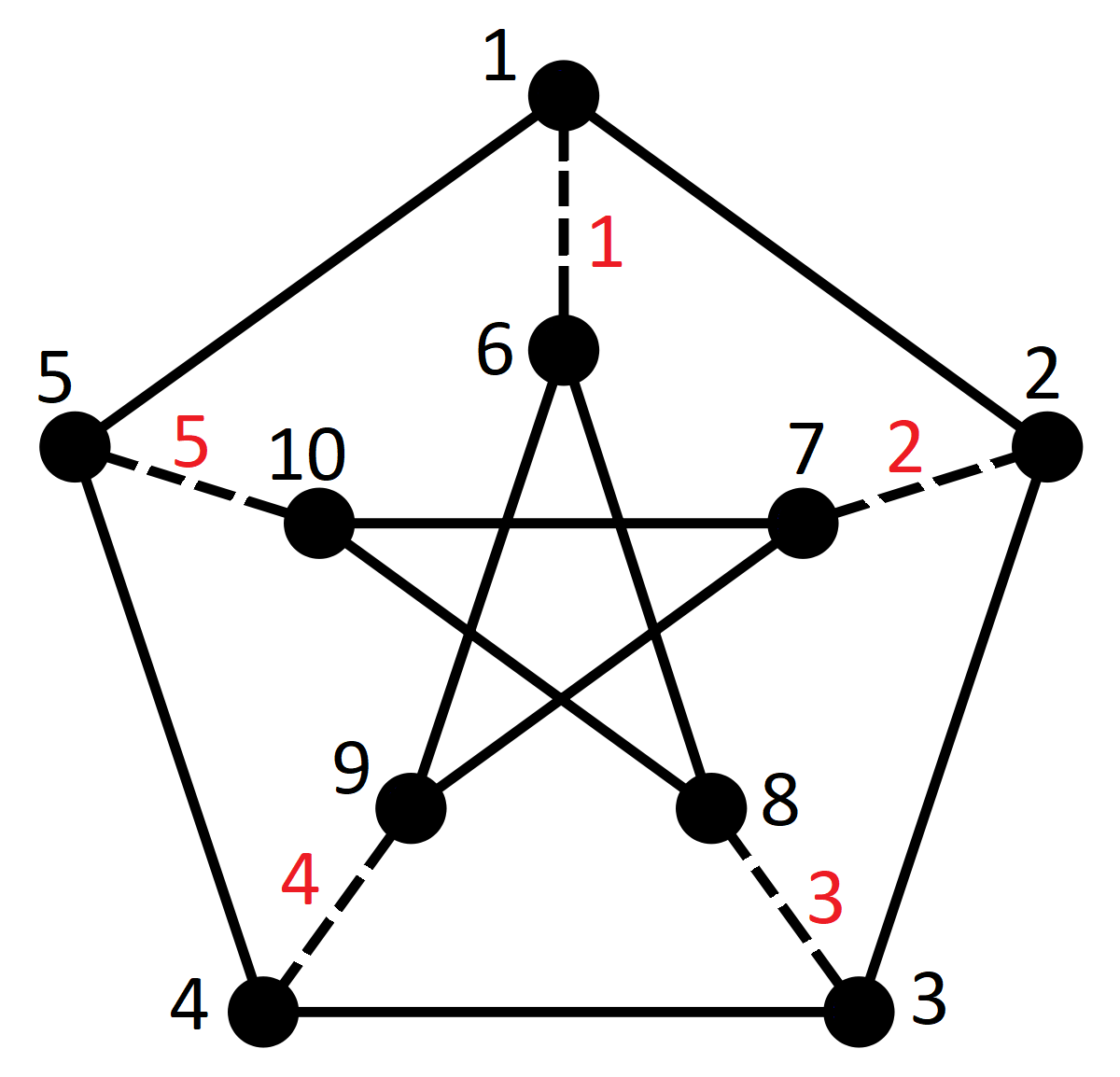}}
		\includegraphics[width=.78\linewidth]{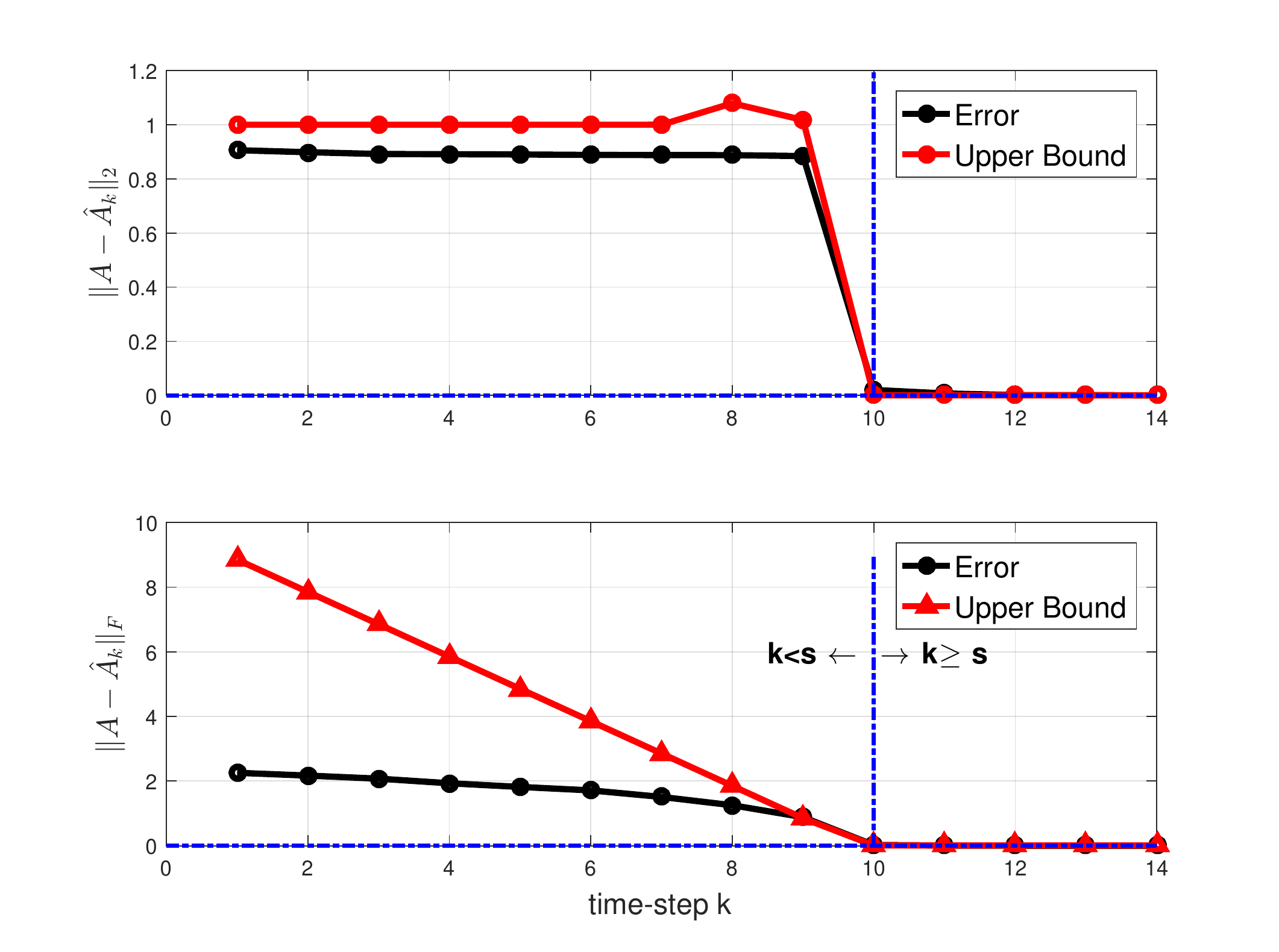}
		\caption{(left) weighted Petersen graph; (right): model regression error the corresponding theoretical error bound}
		\label{fig:plotPetersen}
	\end{figure}

\vspace{-1mm}

\section{CONCLUSION}
\label{sec:conclusion}

In this paper we consider the regression approach for learning linear time-invariant dynamic models from time-series data. In particular, we showed how the richness in the data as well as spectral properties of the model, dictate fundamental bounds on the error obtained from the streaming model regression. Our subsequent works will utilize these insights to provide an active learning mechanism that has the dual role of reducing the regression error in addition to achieving auxiliary control objectives.

%
%	{\color{red} mention robustness for control, effect of noise, uncertainty and time varying dynamics, sparsity promoting}

%. By focusing on these cases, we were able to obtain tighter bounds on the estimation error and also the failure point of our algorithm.
%
%We proposed a least squares based method to learn the parameters of system dynamics, along with the estimation error bound for each iteration. This incremental learning is crucial to allow the user to refine the model in an interactive manner. At this point, our training algorithm does not allow any interaction on the learning process. Ongoing work is directed at designing an online active version of the learning algorithm whereby the algorithm estimates the system parameters incrementally as the control input minimizes the estimation error.
%

	\bibliographystyle{ieeetr}
	\bibliography{citations}

\end{document}